\documentclass[12pt, reqno]{amsart}

\usepackage{amsxtra,amssymb,amsmath,amscd,url,listings}

\usepackage{enumitem}
\newlist{enumth}{enumerate}{1}
\setlist[enumth]{label=\emph{(\arabic*)}, ref=(\arabic*)}

\usepackage[utf8]{inputenc}
\usepackage{eucal}
\usepackage{fullpage}
\usepackage{scrtime}
\usepackage[colorlinks]{hyperref}
\usepackage{datetime}
\usepackage[all]{xy}
\usepackage{graphicx}
\usepackage{tikz-cd}
\usepackage{tikz-cd}
\tikzset{commutative diagrams/arrow style=math font}

\usepackage{mathrsfs}
\usepackage{xspace}

\usepackage{tensor}
\usepackage{braket}

\usepackage{todonotes}
\usepackage{tensor}

\shortdate

\DeclareMathSymbol{A}{\mathalpha}{operators}{`A}%
\DeclareMathSymbol{B}{\mathalpha}{operators}{`B}%
\DeclareMathSymbol{C}{\mathalpha}{operators}{`C}%
\DeclareMathSymbol{D}{\mathalpha}{operators}{`D}%
\DeclareMathSymbol{E}{\mathalpha}{operators}{`E}%
\DeclareMathSymbol{F}{\mathalpha}{operators}{`F}%
\DeclareMathSymbol{G}{\mathalpha}{operators}{`G}%
\DeclareMathSymbol{H}{\mathalpha}{operators}{`H}%
\DeclareMathSymbol{I}{\mathalpha}{operators}{`I}%
\DeclareMathSymbol{J}{\mathalpha}{operators}{`J}%
\DeclareMathSymbol{K}{\mathalpha}{operators}{`K}%
\DeclareMathSymbol{L}{\mathalpha}{operators}{`L}%
\DeclareMathSymbol{M}{\mathalpha}{operators}{`M}%
\DeclareMathSymbol{N}{\mathalpha}{operators}{`N}%
\DeclareMathSymbol{O}{\mathalpha}{operators}{`O}%
\DeclareMathSymbol{P}{\mathalpha}{operators}{`P}%
\DeclareMathSymbol{Q}{\mathalpha}{operators}{`Q}%
\DeclareMathSymbol{R}{\mathalpha}{operators}{`R}%
\DeclareMathSymbol{S}{\mathalpha}{operators}{`S}%
\DeclareMathSymbol{T}{\mathalpha}{operators}{`T}%
\DeclareMathSymbol{U}{\mathalpha}{operators}{`U}%
\DeclareMathSymbol{V}{\mathalpha}{operators}{`V}%
\DeclareMathSymbol{W}{\mathalpha}{operators}{`W}%
\DeclareMathSymbol{X}{\mathalpha}{operators}{`X}%
\DeclareMathSymbol{Y}{\mathalpha}{operators}{`Y}%
\DeclareMathSymbol{Z}{\mathalpha}{operators}{`Z}%

\renewcommand{\leq}{\leqslant}
\renewcommand{\geq}{\geqslant}
 
\def\setminus{\mathchoice
    {\mathbin{\vrule height .72ex width 1.61ex depth -.38ex}}% 12
    {\mathbin{\vrule height .72ex width 1.61ex depth -.38ex}}% 12
    {\mathbin{\vrule height .50ex width 0.85ex depth -.28ex}}%  9
    {\mathbin{\vrule height .20ex width 0.570ex depth -.24ex}}%  7
}

\newcommand{\Cc}{\mathbf{C}}

\newcommand{\Zz}{\mathbf{Z}}
\newcommand{\Pp}{\mathbf{P}}

\newcommand{\Gg}{\mathbf{G}}

\newcommand{\Ff}{\mathbf{F}}

\newcommand{\mcO}{\mathscr{O}}

\def\loccit{loc.\kern3pt cit.{}\xspace}
\def\cf{see\kern.3em}
\def\Cf{See\kern.3em}
\def\eg{e.g.\kern.3em}
\def\ie{i.e.,\ }
\def\resp{\text{resp.}\kern.3em}

\newcommand{\mods}[1]{\,(\mathrm{mod}\,{#1})}

\DeclareMathOperator{\mfm}{\mathfrak{m}}

\renewcommand{\div}{\mathrm{div}}
\DeclareMathOperator{\Div}{Div}

\DeclareMathOperator{\Gal}{Gal}

\DeclareMathOperator{\Res}{Res}

\renewcommand{\rho}{\varrho}

\DeclareMathOperator{\GL}{\mathbf{GL}}
\DeclareMathOperator{\PGL}{\mathbf{PGL}}

\newcommand{\demi}{{\textstyle{\frac{1}{2}}}}

\DeclareMathSymbol{\gena}{\mathord}{letters}{"3C}
\DeclareMathSymbol{\genb}{\mathord}{letters}{"3E}

\theoremstyle{plain}
\newtheorem{theorem}{Theorem}
\newtheorem*{theorem*}{Theorem}
\newtheorem{lemma}[theorem]{Lemma}

\newtheorem{proposition}[theorem]{Proposition}

\theoremstyle{remark}

\theoremstyle{definition}

\newtheorem{remark}[theorem]{Remark}

\newcommand{\mcL}{\mathscr{L}}

\renewcommand{\geq}{\geqslant}
\renewcommand{\leq}{\leqslant}

\begin{document}

\title{Sidon sets in algebraic geometry}

\author{Arthur Forey}
\address[A. Forey]{Univ. Lille, CNRS, UMR 8524 - Laboratoire Paul Painlevé, F-59000 Lille, France} 
  \email{arthur.forey@univ-lille.fr}

\author{Javier Fres\'an}
\address[J. Fres\'an]{CMLS, \'Ecole polytechnique, F-91128 Palaiseau cedex, France}
\email{javier.fresan@polytechnique.edu}

\author{Emmanuel Kowalski}
\address[E. Kowalski]{D-MATH, ETH Z\"urich, R\"amistrasse 101, CH-8092 Z\"urich, Switzerland} 
\email{kowalski@math.ethz.ch}



\begin{abstract}
  We report new examples of Sidon sets in abelian groups arising from
  generalized jacobians of curves, and discuss some of their properties
  with respect to size and structure.
\end{abstract}

\maketitle 

\section{Introduction}

Let $A$ be an abelian group. A subset $S$ of $A$ is called a \emph{Sidon set} if $S$ does not contain non-trivial additive quadruples; that is, if any solution $(x_1,x_2,x_3,x_4)\in S^4$ of the equation
\begin{equation}\label{eq-sidon}
  x_1+x_2=x_3+x_4
\end{equation}
satisfies $x_1\in\{x_3,x_4\}$ (see,
e.g.,~\cite[\S\,1]{eberhard-manners}). In other words, up to
transposition an element of~$A$ is in at most one way the sum of two
elements of~$S$.

We will explain how to construct a range of new examples of Sidon sets
using the theory of commutative algebraic groups. In fact, we sometimes
most naturally obtain a slight variant: given an element~$a$ of~$A$, we
say that a subset $S$ of~$A$ is a \emph{symmetric Sidon set with
  center~$a$} if $S=a-S$ and the solutions to equation \eqref{eq-sidon}
satisfy $x_1\in\{x_3,x_4\}$ or~$x_2=a-x_1$ (we will explain in
Remark~\ref{rm-center} that the center is unique if~$S$ is not
empty). Choosing (arbitrarily) one element of $\{x,a-x\}$ as $x$ varies
over elements of~$S$ with $2x\not=a$ leads to a Sidon set of size
about~$|S|/2$ if $S$ is finite and~$A$ is without $2$-torsion, but there
is usually no natural choice.

\begin{theorem}\label{th-classification}
  Let $k$ be a field and let~$C$ be a smooth projective geometrically connected curve of genus~$g$ over~$k$. Let~$\mfm$ be an
  effective divisor on~$C$ and $J_{\mfm}$ the associated generalized jacobian, which is a commutative algebraic group of dimension $g+\max(\deg(\mfm)-1,0)$.  Let~$\delta$ be a divisor of degree~$1$
  on~$C$ whose support does not intersect that of~$\mfm$.
  Let~$s\colon C\setminus\mfm\to J_{\mfm}$ be the
  morphism induced by the map $x\mapsto (x)-\delta$ on
  divisors.
 
If $\dim(J_{\mfm})\geq 2$, then $s((C\setminus \mfm)(k))$ is either a Sidon set or a
  symmetric Sidon set in $J_{\mfm}(k)$. 

If, moreover, $(C\setminus\mfm)(k)$ is non-empty, then it is a symmetric Sidon set if and only if one of the following conditions
  hold:
  \begin{enumth}
  \item $g=1$ and~$\deg(\mfm)=2$; in this case, writing $\mfm=(p)+(q)$
    (where~$p$ and~$q$ are not necessarily $k$-points of $C$, but the
    divisor $\mfm$ is assumed to be defined over~$k$), the center
  of~$s((C\setminus \mfm)(k))$ is the common value of $s(x)+s(p+q-x)$
  for any $x\in (C\setminus \mfm)(k)$.
  \item $g\geq 2$, the curve~$C$ is hyperelliptic, and either
    $\deg(\mfm)\leq 1$ or $\mfm=(p)+(i(p))$ for
    some~$p\in C$, where~$i$ is the hyperelliptic involution on $C$. In both of these
    cases, the center of~$s((C\setminus \mfm)(k))$ is the common
    value of $s(x)+s(i(x))$ for any $x\in (C\setminus \mfm)(k)$.
  \end{enumth}
\end{theorem}

A concrete description of the abelian groups $J_{\mfm}(k)$ will be presented in Section~\ref{sec-jacobians}. 

Figure~\ref{fig-quadruple} illustrates the symmetric Sidon set obtained from a curve~$C$ of genus~$1$, viewed
  as a plane cubic curve, and a divisor $\mfm=(p)+(q)$ supported on two distinct $k$-points. It
  displays a configuration of points $(x_1,\ldots,x_4) \in C(k) \setminus \{p, q\}$ such that
  $(s(x_1),\ldots,s(x_4))$ is a non-trivial additive quadruple in $J_{\mfm}(k)$. 

The case $\mfm=0$ of Theorem~\ref{th-classification} was
  known to N. Katz (see Section~\ref{sec-appl}).
    
  \begin{figure}
      \includegraphics[width=3in]{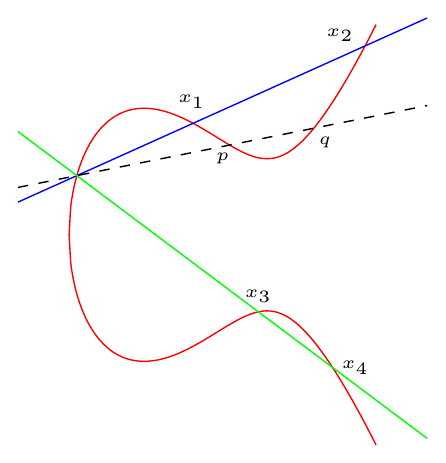}
    \caption{A non-trivial additive quadruple for
      $J_{(p)+(q)}$}\label{fig-quadruple}
  \end{figure}

Before proving Theorem~\ref{th-classification}, we will recall the
definitions of the generalized jacobians, and comment on these examples
of Sidon sets in comparison with the current literature. We also briefly
survey some surprising applications of Sidon sets in algebraic geometry.

\begin{remark}\label{rm-center}
  (1) It is possible that $S$ be both a Sidon set and a symmetric
  Sidon set, but this only happens if~$S$ is empty or if $S=\{x,a-x\}$
  for some $a$ and~$x$ in $A$.

  (2) Let~$S\subset A$ be a non-empty symmetric Sidon set. We claim
  that its center is unique. Indeed, this is straightforward to check if $S$ is of the form $\{x, a-x\}$. Otherwise, $S$ is not a Sidon set, so that there exist elements $x_1, \ldots, x_4 \in S$ such that $x_1+x_2=x_3+x_4$ and $x_1 \notin \{x_3, x_4\}$. Then \textit{any} center $a$ is equal to $x_1+x_2$, and hence all centers are equal.  
  \end{remark}

\subsection*{Notation}
\label{sec-conventions}

Given complex-valued functions $f$ and $g$ defined on a set $S$, we
write $f \ll g$ if there exists a real number $C \geq 0$ (which is then called an
``implicit constant'') such that the inequality $|f(s)|\leq C g(s)$
holds for all $s \in S$.

\subsection*{Acknowledgments} We thank K. Soundararajan for pointing out
to us the terminology ``Sidon sets'', and S. Eberhard and F. Manners for
sharing with us their work~\cite{eberhard-manners} on dense Sidon
sets. We also thank C. Bortolotto for pointing out a small slip when
passing from symmetric Sidon sets to Sidon sets. We thank very warmly
the referee, whose highly perspicuous report was a model of the
kind. During the preparation of this work, A.\,F. was supported by the
SNF Ambizione grant PZ00P2\_193354 and J.\,F. was partially supported by
the grant ANR-18-CE40-0017 of the Agence Nationale de la Recherche.

\section{Generalized jacobians}\label{sec-jacobians}

In this preliminary section, we briefly recall the group structure on the set of points of generalized jacobians (see Serre's
book~\cite{serre} for a complete account of the theory, in particular for the further
structure of algebraic group that they carry).

Let~$k$ be an algebraically closed field.  Let~$C$ be a smooth
projective geometrically connected algebraic curve over~$k$. The
group~$\Div(C)$ of \emph{divisors} on~$C$ is the free abelian group with
basis given by the $k$-points of~$C$; we denote by $(p)$ the basis
element corresponding to~$p\in C(k)$.

The \emph{degree}~$\deg(D)$ of a divisor $D=\sum n_p (p)$ is the sum $\sum n_p$ of the coefficients in its expression as a $\Zz$\nobreakdash-linear combination of the basis elements~$(p)$.  One
views~$\Div(C)$ as an ordered abelian group, with~$D\geq 0$ if and only
if all coefficients are non-negative integers (in which case $D$ is
called an \emph{effective} divisor).  The \emph{support} $|D|$ of a divisor~$D$ is the set
of~$p\in C(k)$ such that $n_p$ is non-zero, and two
divisors are said to be \emph{coprime} if they have disjoint supports. We write
$C\setminus D$ for the open subvariety of~$C$ obtained by removing~$|D|$.

If $f\colon C\to \Pp^1$ is a non-constant function, the divisor
$\div(f)$ of~$f$ is defined to be the sum of~$p\in C(k)$ such that $f(p)=0$, with multiplicity, minus the sum of $p\in C(k)$ such that~$f(p)=\infty$, with multiplicity. It is a standard fact that
$\deg(\div(f))=0$. 

Fix an effective divisor $\mfm=\sum n_p (p)$ on~$C$, called a
\emph{modulus}. Let~$\Div_{0,\mfm}(C)$ be the subgroup of
$\Div(C)$ consisting of divisors of degree~$0$ which are coprime to~$\mfm$. This contains a subgroup $P_{\mfm}(C)$ whose
elements are the divisors of non-constant functions~$f$ satisfying~$v_p(f-1)\geq n_p$ for every~$p$ in the support of~$\mfm$ with multiplicity~$n_p\geq 1$, where~$v_p$ is the valuation at~$p$ (order of zero or pole at~$p$). In particular, the divisor of such a function is coprime to $\mfm$.

The group of $k$-points of the \emph{generalized jacobian} associated to~$C$ and~$\mfm$ is then defined as
\[
J_{\mfm}(k)=\Div_{0,\mfm}(C)/P_{\mfm}(C).
\]

The special case of the trivial modulus $\mfm=0$ is particularly important: the
corresponding quotient is the group of $k$-points of the classical jacobian variety $J$ of~$C$, which is an abelian variety over~$k$; its dimension is the \emph{genus}~$g$ of~$C$. In general, $J_{\mfm}(k)$ is the group of $k$\nobreakdash-points of a commutative algebraic group $J_{\mfm}$ of dimension
$g+\max(\deg(\mfm)-1,0)$ over $k$ (see \hbox{\cite[V.1.3\,and\,V.1.6]{serre}}). It  fits into an extension
\begin{equation}\label{eqn:extensionjacobian} 
0 \longrightarrow \bigg(\prod_{p \in |\mathfrak{m}|} \mathrm{U}_p / \mathrm{U}_p ^{(n_p)} \bigg)\bigg/ \Delta \longrightarrow J_{\mfm}(k) \longrightarrow J(k) \longrightarrow 0, 
\end{equation} where $\mathrm{U}_p$ denotes the multiplicative group
of functions that do not vanish at $p$ (units), $\mathrm{U}_p^{(n_p)}$
the subgroup of those satisfying $v_p(f-1) \geq n_p$, and $\Delta$ the
diagonal subgroup of non-zero constant functions $(\lambda, \dots,
\lambda)$. According to~\cite[V.3.15\,and\,V.3.16]{serre}, each group
$\mathrm{U}_p / \mathrm{U}_p ^{(n_p)}$ is isomorphic to $k^\times
\times k^{n_p-1}$ if $k$ has characteristic zero, and to $k^\times
\times \prod W_{r_i}(k)$ if $k$ has positive characteristic $\ell$, where the product runs over integers $1 \leq i \leq n_p-1$ coprime to $\ell$ and $W_{r_i}(k)$ are the $\ell$-typical Witt vectors of length the smallest integer $r_i$ satisfying $\ell^{r_i} \geq n_p/i$. 

\begin{remark}
  The formula for the dimension of the generalized jacobian shows that the condition $\dim(J_{\mfm})\geq 2$
  of Theorem~\ref{th-classification} only excludes the cases where $g=0$
  and~$\deg(\mfm)\leq 2$ or~$g=1$
  and~$\deg(\mfm)\leq 1$. The corresponding generalized
  jacobians are as follows:
  \begin{itemize}
    \item If~$g=0$ and~$\deg(\mfm)\leq 1$, then $C$ is the
      projective line and~$J_{\mfm}$ is the trivial group.
    \item If~$g=0$ and~$\deg(\mfm)=2$, then $C$ is the
      projective line and~$J_{\mfm}$ is either the additive
      group (if $\mfm$ is a single point with multiplicity~$2$)
      or the multiplicative group over~$k$ (if~$\mfm$ consists of two
      points).
    \item If~$g=1$ and~$\deg(\mfm)\leq 1$, then~$C$ can be
      identified with an elliptic curve after fixing an origin,
      and~$J_{\mfm}$ is isomorphic to this elliptic curve.
  \end{itemize}
\end{remark}

We have now described the groups appearing in
Theorem~\ref{th-classification} when~$k$ is algebraically
closed. If $k$ is only perfect (\eg a finite field), we can define $J_{\mfm}(k)$ by Galois descent. Namely, we fix an algebraic closure $\bar{k}$ of $k$, and consider the action of the absolute Galois group $\mathrm{Gal}(\bar{k}/k)$ on $C(\bar{k})$, which extends by linearity to an action on the group $\Div(C_{\bar{k}})$ of divisors over $\bar{k}$. We define the group $\Div(C)$ of divisors over $k$ as the fixed points under this action. Given a modulus~$\mfm$ on $C$, \ie an effective divisor over $k$, the subgroups $\Div_{0, \mfm}(C_{\bar{k}})$ and $P_{\mfm}(C_{\bar{k}})$
are stable under the action of $\Gal(\bar{k}/k)$, hence an induced action on the quotient $J_{\mfm}(\bar{k})$. The group $J_{\mfm}(k)$ is then the subgroup of~$J_{\mfm}(\bar{k})$ consisting of the fixed points under the Galois group. The most interesting case from the point of view of classical Sidon set theory is that of finite field $k$, in which the above construction amounts to considering fixed points under the Frobenius automorphism $x\mapsto x^{|k|}$ of~$\bar{k}$. Then $J_{\mfm}(k)$ is a finite group.

\begin{remark} Examples of divisors over $k$ include, of course, linear combinations of $k$-points of $C$, but a
  $k$-divisor of~$C$ is \emph{not necessarily} of this form. For instance, for~$C=\Pp^1$ and a non-zero polynomial~$f\in k[X]$, the divisor of~$f$ is always a $k$-divisor on~$C$, although in
  general not all roots of~$f$ belong to~$k$. This simple remark will
  play a role in the next section.
\end{remark}

\section{Algebraic Sidon sets over finite fields}\label{sec-remarks}

A classical problem in additive combinatorics is to construct large
Sidon subsets of finite groups. Thus, while it is enough to prove
Theorem~\ref{th-classification} for algebraically closed fields, we
are particularly interested in the case where~$k$ is finite, and hence
$J_{\mathfrak{m}}(k)$ is a finite group.  We therefore investigate the
size and apparent structure of the finite Sidon sets we have
constructed, assuming that~$k$ is a finite field with a fixed
algebraic closure $\bar{k}$.

Given a curve~$C$ of genus $g$ and a modulus~$\mfm$ over~$k$
satisfying $g+\max(\deg(\mfm)-1,0) \geq 2$, the morphism
$s \colon C\setminus \mfm \to J_{\mfm}$ is an embedding, so that
Theorem~\ref{th-classification} provides a Sidon set or a symmetric
Sidon set~$S=s((C\setminus\mfm)(k))$ of size $|(C\setminus \mfm)(k)|$
in the abelian group~$A=J_{\mfm}(k)$. On the one hand, from the
Hasse--Weil bound on the number of points of curves over finite
fields~\cite{weil}, one gets the estimate
\[
  |k|-2g\sqrt{|k|}+1-\deg(\mfm) \leq |S|\leq |k|+2g \sqrt{|k|}+1.
\]

On the other hand, the extension structure
\eqref{eqn:extensionjacobian} of $J_{\mfm}(k)$ along with the Riemann
hypothesis for abelian varieties and tori over finite fields, in the
form of the estimates
\[
  (\sqrt{|k|}-1)^{2g} \leq |J(k)|\leq (\sqrt{|k|}+1)^{2g}, \quad
  (|k|-1)^d \leq |T(k)| \leq (|k|+1)^d
\]
for a $d$-dimensional torus $T$ (see \eg \cite[Thm.\,19.1]{milne-ab}
and \cite[Prop.\,3.3.5]{carter}), yield
\[
  (|k|-1)^{\max(\deg(\mfm)-1,0)}(\sqrt{|k|}-1)^{2g}\leq |A|\leq
  (\sqrt{|k|}+1)^{2g} (|k|+1)^{\max(\deg(\mfm)-1,0)}.
\]

Thus, when~$|k|$ is large, the set $S$ has size about
$|A|^{1/\dim(J_{\mfm})}$. The densest sets will therefore appear
when $\dim(J_{\mfm})=2$. This happens in the following cases:
\begin{enumerate}
\item $g=0$ and $\deg(\mfm)=3$;
\item $g=1$ and $\deg(\mfm)=2$;
\item $g=2$ and $\deg(\mfm)\leq 1$.
\end{enumerate}

Note that in the second and third cases, Theorem~\ref{th-classification}
also states that we obtain a \emph{symmetric} Sidon set, and not a Sidon
set (because any curve of genus~$2$ is
hyperelliptic~\cite[Prop.~4.9]{Liu}), so that we get from this
construction Sidon sets of size about~$\sqrt{|A|}/2$ by
``desymmetrizing'' (using the fact that the size of the $2$-torsion
group of $J_{\mfm}(k)$ is bounded by $2^4$ in all these cases).

We consider these three cases in turn.

(1) For $g=0$ and~$\deg(\mfm)=3$, the curve~$C$ is isomorphic to~$\Pp^1$
(since any non-degenerate quadratic form in three variables has a non-trivial zero by the Chevalley--Warning theorem; this would not necessarily be true over an arbitrary field).  We can
restrict our attention to a few special cases of~$\mfm$ using the action
of the automorphism group~$\PGL_2(k)$ on~$C(k)$, which induces an action
on the group of~$k$-divisors preserving multiplicities. 

\begin{lemma}
The action
  of~$\PGL_2(k)$ on effective $k$-divisors of degree~$3$ on~$\Pp^1$ has
  five orbits, represented by
  \begin{gather*}
    \mfm_1=(0)+(1)+(\infty),\quad\quad
    \mfm_2=(0)+2(\infty),\quad\quad
    \mfm_3=3(\infty),\\
    \mfm_4=(a)+(b)+(c),\quad\quad
    \mfm_5=(\alpha)+(\beta)+(\infty),
  \end{gather*}
  where $a,b,c$ are the roots of an irreducible monic cubic polynomial $f_3\in k[X]$,  and~$\alpha,\beta$ are the roots of an irreducible monic quadratic polynomial $f_2\in k[X]$.  
\end{lemma}

\begin{proof}
  Since~$\PGL_2(k)$ acts $3$-transitively on~$\Pp^1(k)$, the orbits of
  the divisors $\mfm_1$, $\mfm_2$ and~$\mfm_3$ are, respectively, the
  divisors whose support is contained in~$\Pp^1(k)$ and consists of
  three distinct points, two points, or a single point.

  Now let~$\mfm$ be an effective $k$-divisor of degree~$3$ such that at least one
  point~$x$ of the support of~$\mfm$ does not belong to~$\Pp^1(k)$. Since, for any such~$x$, all its Galois conjugates are
  also in the support with the same multiplicity (because $\mfm$ is
  Galois-invariant), we see that~$x$ must generate either the
  extension~$k_3$ of degree~$3$ of~$k$ in~$\bar{k}$, or the extension $k_2$ of
  degree~$2$.

  In the former case, $\mfm$ is equal to
  $(x)+(x^{|k|})+(x^{|k|^2})$. But there exists~$\gamma\in \PGL_2(k)$
  such that~$\gamma(x)=a$ (because $\GL_2(k)$ acts on $k_3\setminus k$
  with stabilizers given by the center, so acts transitively since
  $|\mathbf{GL}_2(k)|=|k^{\times}|\, |k_3\setminus k|$) and then
  $\gamma(\mfm)=\mfm_4$ since the Frobenius commutes with~$\gamma$.

  Finally, if~$x$ generates the quadratic extension $k_2$, then $\mfm$ is of the form 
  $(x)+(x^{|k|})+(y)$ for some $y \in \Pp^1(k)$, and from the transitivity of the action of~$\PGL_2(k)$ on~$\Pp^1(k)$ and that of $\GL_2(k)$ on $k_2$ one sees that $\mfm$ and $\mfm_5$ lie in the same orbit. 
\end{proof}

We obtain this way five Sidon sets~$S_i\subset A_i=J_{\mfm_i}(k)$ for
$1\leq i\leq 5$, of sizes
$$
|S_1|=|k|-2,\quad |S_2|=|k|-1,\quad |S_3|=|k|,\quad |S_4|=|k|+1,\quad
|S_5|=|k|.
$$
\par
Moreover, one can easily check that there are isomorphisms of abelian
groups
\begin{gather*}
  A_1\simeq (k^{\times})^2,\quad A_2\simeq k^{\times}\times k, \\
A_3\simeq k^2,\text{ if $k$ has characteristic $\geq 3$}\\
A_3\simeq W_2(k),\text{ if $k$ has characteristic $2$},\\
A_4\simeq k_3^{\times}/k^\times,\quad A_5\simeq k_2^{\times},
\end{gather*}
where $k_3$ and~$k_2$ are respectively the cubic and quadratic extensions
of~$k$ inside $\bar{k}$ and~$W_2(k)$ is the group of Witt vectors of length~$2$ (which in characteristic~$2$ is a non-trivial extension of~$k$ by~$k$;
see~\cite[V.16]{serre}). The groups $A_4$ and~$A_5$ appear as the groups
of $k$-points of the $2$-dimensional non-split $k$-tori $J_{\mfm_4}\simeq \Res_{k_3 /k}(\Gg_m)/\Gg_m$
and~$J_{\mfm_5}\simeq \Res_{k_2/k}(\Gg_m)$.

It is not difficult to see further that $S_1$, $S_2$, $S_3$ can be
identified, respectively, with
\begin{gather*}
  S_1=\{(x,1-x)\in (k^{\times})^2\,\mid\, x\in k^{\times},\ x\not=1\},
  \\
  S_2=\{(x,x)\in k^{\times}\times k\,\mid\, x\in k^{\times}\},\\
  S_3=\{(x,x^2)\in k^2\,\mid\, x\in k\},\text{ if $k$ has
    characteristic~$\geq 3$}
\end{gather*}
(see~\cite[Rem.\,9.13\,(2)]{ffk}).  These are very classical examples
of Sidon sets; they appear in the paper~\cite{eberhard-manners} of
Eberhard and Manners as Constructions 5, 4 and~1, respectively, and are
due to Erd\H os--Turán, Spence and Hugues (with~$S_2$ also discovered
independently by Ruzsa and~$S_3$ by Cilleruelo). One can also check that $S_4$ and~$S_5$ correspond to Constructions 2
and~3 of \loccit, which are due to Singer and
Bose, respectively.

All these are Sidon sets of size approximately $\sqrt{|A|}$. Thus, we recover
``uniformly'' the five main examples of dense Sidon sets discussed by
Eberhard and Manners. This construction appears to be very different from their own uniform interpretation, where the groups~$A_i$ arise as maximal abelian subgroups of $\PGL_3(k)$ and the Sidon sets take the form
$$
S=\{g\in A\,\mid\, p\in g(\ell)\}
$$
for some line $\ell$ and some point~$p$ in~$\Pp^2(k)$
(see~\cite[\S\,2-3]{eberhard-manners}).

(2) For $g=1$, we have a curve of genus~$1$. It is classical that, over
a finite field, such a curve always has a $k$-rational point, and one
can take this as origin to view the curve as an elliptic curve. In
particular, the set $C(k)$ is then also a finite abelian group. The general structure of the generalized jacobians from \eqref{eqn:extensionjacobian} specializes in this case to a short exact sequence 
$$
0\to B\to J_{\mfm}(k)\to C(k)\to 0, 
$$
where the abelian group~$B$ is given by: 
\begin{align*}
  B&=k\text{ if } \mfm=2(p)\text{ for some } p\in C(k),\\
  B&=k^{\times}\text{ if } \mfm=(p)+(q)\text{ for some } p\not=q\text{
    in
  } C(k),\\
  B&=k_2^{\times}/k^{\times}\text{ if } \mfm=(p)+(q)\text{ for some Galois-conjugate }
  p\not=q\text{ in } C(k_2). 
\end{align*}
Since~$|k|-1\leq |B|\leq |k|+1$ (and $J_{\mfm}(k)$ has at most $8$
points of $2$-torsion), the desymmetrized Sidon sets have size about
$\sqrt{|A|}/2$.

(3) For $g=2$ and $\deg(\mfm)\leq 1$, the generalized jacobians are all
isomorphic to the classical jacobian of $C$, and again we obtain Sidon
sets of size about $\sqrt{|A|}/2$.

All these examples are rather dense Sidon sets. Cases (2) and (3) are
seemingly of a different type than previous examples, which is of
interest in the context of the existing speculation that ``sufficiently
large'' Sidon sets in finite abelian groups should have some kind of
algebraic structure (see \eg the blog post~\cite{gowers} of T. Gowers,
and the comments there). As already mentioned, Eberhard and
Manners~\cite{eberhard-manners} have classified in a uniform way all
known examples of Sidon sets $S$ with~$|S|\sim \sqrt{|A|}$ using finite
projective planes.  Our constructions show that there is a much wider
variety of examples of Sidon sets of size~$\sqrt{|A|}/2$ than previously
reported, and exhibit the following features which, to the best of our knowledge, were previously unknown: 

\begin{itemize}
\item Any classification of Sidon sets of size at least~$\sqrt{|A|}/2$
  will have to be sophisticated enough to account for jacobians of
  curves of genus~$2$ as well as generalized jacobians of dimension~$2$ coming from elliptic curves;
\item There are natural ``continuous'' families of Sidon sets of size
  about $\sqrt{|A|}/2$, up to ``isomorphism''. Namely, we note that
  the space of hyperelliptic curves of genus~$2$ over a given field~$k$,
  up to isomorphism, is three-dimensional; the space of elliptic curves
  is one-dimensional, each giving rise to a one-parameter family of
  generalized jacobians (for $\mfm=2(p)$, it is not difficult to
  see that all generalized jacobians are isomorphic as~$p$ varies, and
  for $\mfm=(p)+(q)$, one can check that~$J_{(p)+(q)}$ is
  isomorphic to~$J_{(0)+(q-p)}$). Although one might object that maybe
  distinct curves (as geometric objects) would give rise to
  ``isomorphic'' finite Sidon sets, this is certainly not the case, at
  least in a naive sense (e.g. because many different finite abelian
  groups arise as $J(k)$ for the jacobian $J$ of a curve of genus~$2$
  and a fixed large finite field $k$).
\item All our examples are obtained as the intersection of an
  \emph{infinite} Sidon set, namely
  \[
  s((C\setminus \mfm)(\bar{k}))\subset J_{\mfm}(\bar{k}),
  \] with the
  finite subgroup~$J_{\mfm}(k)$. (Note that this applies also to
  examples where the modulus is not defined over~$k$; for instance,
  this happens for the sets $S_4$ and~$S_5$ above, in which case this
  feature is not apparent from the classical constructions of Bose and
  Singer, or those of Eberhard--Manners.)

  An intriguing comparison suggests itself with infinite Sidon sets of
  integers, which are known to be \emph{less} dense in segments than the
  densest finite Sidon sets of integers. Indeed, a result of Erd\H os states
  that
  $$
  \liminf_{n\to +\infty} \frac{(\log N)^{1/2}}{N^{1/2}} |\{n\in S\,\mid\,
  0\leq n\leq N\}|<+\infty
  $$
  for any Sidon set~$S\subset \Zz$ (see,
  e.g.,~\cite[p.\,89,\,Th.\,8]{halberstam-roth}). The referee pointed out
  that the proof leads to a precise upper-bound $\leq \sqrt{40}$.
\end{itemize}

\section{Some finite abelian groups with large Sidon sets}

Babai and Sós~\cite[Th.\,4.2]{babai-sos} considered the problem of
finding ``large'' Sidon subsets in arbitrary finite abelian groups. It
should be noted before stating their results that they use a slightly
different definition of Sidon sets (see~\cite[Def.\,1.1]{babai-sos}),
which also allows for solutions of~(\ref{eq-sidon}) with
$x_1=x_2$ and $x_3=x_4$, which are not trivial in the sense of our
definition unless also~$x_1=x_3$. (In other words, a Sidon set can contain distinct elements $x_1, x_3$ satisfying $2x_1=2x_3$.) This
definition coincides with ours if $A$ has trivial $2$-torsion but not
in general: for instance, a vector space over~$\Ff_2$ does not
contain Sidon sets of size $\geq 2$ with our definition. Babai and Sós
proved (using a probabilistic argument) that any finite abelian
group~$A$ contains a set~$S$ with their property such
that~$|S|\gg |A|^{1/3}$; as far as we know, this remains the best general lower bound. Our results lead to new families of finite abelian groups~$A$ in which
Sidon sets~$S$ with $|S|\gg |A|^{1/2}$ are known to exist.

\begin{proposition}\label{pr-4}
  Let~$j\geq 1$ be an integer and $(n_i)_{1\leq i\leq j+1}$ a
  finite sequence of integers~$\geq 2$.  Suppose that there exist a prime number~$p$ and an integer~$n$ coprime
  to~$p$ such that
  \begin{gather*}
    p\leq n_i\text{ for } 1\leq i\leq j,\quad n\leq n_{j+1},\\
    |p^j+1-n|\leq 2p^{j/2}.
  \end{gather*}  Then there exists a Sidon set
  $$
  S \subset A=\prod_{1\leq i\leq j+1} \Zz/n_i\Zz
  $$
such that $|S|\gg (p/2)^j$,
  where the implicit constant is absolute.
\end{proposition}

Roughly speaking, this means that if we have a prime number~$p$ and
integers~$j$ and~$n$ with~$n$ of size about $p^j$, then a finite abelian
group~$A$ which is ``close to'' the group
\begin{equation}
  \label{eq-e-sharp}
  (\Zz/p\Zz)^j\times \Zz/n\Zz
\end{equation}
(in some sense) contains a Sidon set of size~$\gg |A|^{1/2}$. By
contrast, the results of Babai and Sós~\cite[Prop.\,5.3]{babai-sos} give
such a lower bound for groups ``close to''
$$
(\Zz/p\Zz)^j.
$$
For groups like~(\ref{eq-e-sharp}), the bound of Babai and Sós is of
size~$|A|^{1/4}$, which is therefore worse than the general
probabilistic lower bound $|A|^{1/3}$.

\begin{proof}[Proof of Proposition~\ref{pr-4}]
  The argument is similar to the proof
  of~\cite[Prop.\,5.3]{babai-sos}.  Let~$k$ be a field with~$p^j$
  elements. Pick an elliptic curve~$E$ over $k$ such that $E(k)$ is
  cyclic of order~$n$ (which is possible, by independent work of
  Rück~\cite{rueck} and Volloch~\cite{voloch}), set
  $\widetilde{E}=E \setminus \{0_E\}$ and consider the generalized
  jacobian $E^{\sharp}=J_{2(0_E)}$. Choose a $k$-rational divisor of
  degree~$1$ to define the embedding
  $s\colon \widetilde{E}\to E^{\sharp}$. We also fix a group
  isomorphism
  $$
  E^{\sharp}(k)\to (\Zz/p\Zz)^j\times \Zz/n\Zz
  $$
  (which exists since there is an exact sequence
  $$
  0\to k\to E^{\sharp}(k)\to E(k)\simeq \Zz/n\Zz\to 0,
  $$
  and the assumption that $n=|E(k)|$ is coprime to $p$ implies that this
  exact sequence is split). We denote by
  $t\colon \widetilde{E}(k)\to (\Zz/p\Zz)^j\times \Zz/n\Zz$ the
  composition of~$s$ and such an isomorphism. The image
  $t(\widetilde{E}(k))$ is a symmetric Sidon set of size~$n-1$ by
  Theorem~\ref{th-classification}.

  % that all extensions of $E(k)$ by $k$ are split.
  % Consider the composition
  % \begin{equation*}
  %   \begin{tikzcd}
  %     \widetilde{E}(k)\arrow[r,hook, "s"] & E^{\sharp}(k)
  %     \arrow[r,"\sim"] & (\Zz/p\Zz)^j\times \Zz/n\Zz
  %   \end{tikzcd}
  % \end{equation*}
  % where the existence of the isomorphism follows from the exact sequence
  % $$
  % 0\to k\to E^{\sharp}(k)\to E(k)\to 0
  % $$
  % and the assumption that $n=|E(k)|$ is coprime to $p$, so that all extensions of $E(k)$ by $k$ are split.

  For any integer~$q\geq 1$ and any~$x\in \Zz/q\Zz$, we denote
  by~$\dot{x}$ the integer such that~$0\leq \dot{x}\leq q-1$
  and~$x\equiv \dot{x}\mods{q}$. By an \emph{interval} in~$\Zz/q\Zz$,
  we mean the image modulo~$q$ of an interval $\{a,a+1,\ldots, b\}$,
  where~$0\leq a\leq b\leq q-1$; the integer~$a$ is called the
  \emph{origin} of the interval.

  % Consider the composition
  % \begin{equation*}
  %   \begin{tikzcd}
  %     \widetilde{E}(k)\arrow[r,hook, "s"] & E^{\sharp}(k)
  %     \arrow[r,"\sim"] & (\Zz/p\Zz)^j\times \Zz/n\Zz
  %   \end{tikzcd}
  % \end{equation*}
  % where the existence of the isomorphism follows from the exact sequence
  % $$
  % 0\to k\to E^{\sharp}(k)\to E(k)\to 0
  % $$
  % and the assumption that $n=|E(k)|$ is coprime to $p$, so that all
  % extensions of $E(k)$ by $k$ are split.
  
  By the pigeon-hole
  principle, there is a choice of intervals $I_i$
  for~$1\leq i\leq j$ (\resp $I$) with $I_i\subset \Zz/p\Zz$ for
  $1\leq i\leq j$ (\resp $I\subset \Zz/n\Zz$) and
  $|I_i|=\lfloor p/2\rfloor$ (\resp $|I|= \lfloor n/2\rfloor)$, so
  that
  $$
  |t(\widetilde{E}(k))\cap X|\geq \frac{|\widetilde{E}(k)|}{2^{j+1}},
  $$
  where we denote $X=I_1\times\cdots\times I_j\times I$.
  % \prod_{1\leq i\leq j}I_i\times
  % I\Bigr)

  (We note in passing that a fairly simple appeal to the Riemann
  hypothesis over finite fields and discrete Fourier analysis
  on~$E^{\sharp}(k)$ shows that this will be valid for all choices of
  intervals of this size, provided $p$ is large enough and the constant
  $2^{-(j+1)}$ is replaced, say, by $2^{-(j+2)}$.)

  Let~$\alpha_j$ (resp. $\alpha$) be the origin of~$I_j$
  (resp. of~$I$). 
  Define an injective map
  $$
  \varphi\colon (\Zz/p\Zz)^j\times \Zz/n\Zz\to \Zz^{j+1}
  $$
  by the assignment
  $(a_1+p\Zz,\ldots,a_j+p\Zz,a+n\Zz)\mapsto
  (\dot{a}_1-\alpha_1,\ldots, \dot{a_j}-\alpha_j,\dot{a}-\alpha)$. The
  restriction of~$\varphi$ to $X$ is a \emph{Freiman isomorphism} of
  order~$2$ (i.e., for $x_1, \dots, x_4$ in~$X$, we have
  $x_1+x_2=x_3+x_4$ if and only
  if~$\varphi(x_1)+\varphi(x_2)=\varphi(x_3)+\varphi(x_4)$).

  In addition, by the assumptions $p\leq n_i$ and $n\leq n_{j+1}$, as well as the
  condition on the intervals, the image of~$E^{\sharp}(k)$ by $\varphi$
  (resp. the image of~$X$ by $\varphi$) is contained in the set
  \begin{gather*}
    \{0,\ldots, p-1\})^j\times \{0,\ldots, n-1\},\\
    (\resp \{0,\ldots, \lfloor p/2\rfloor\}^j\times
    \{0,\ldots, \lfloor n/2\rfloor\}).
  \end{gather*}

  In particular, the restriction of the canonical projection
  $\pi\colon \Zz^{j+1}\to A$ to the image of~$\varphi$ is injective,
  and the restriction of~$\pi$ to~$X$ is also a Freiman isomorphism of
  order~$2$.

  Thus, we have a Freiman isomorphism $\pi\circ \varphi$ of order~$2$
  from~$X$ to~$A$, and since $t(\widetilde{E}(k))$ is a symmetric Sidon
  set in~$(\Zz/p\Zz)^j\times\Zz/n\Zz$, we conclude that the set
  $$
  (\pi\circ\varphi)(t(\widetilde{E}(k))\cap X)\subset A
  $$
  is a symmetric Sidon set in~$A$ of size
  $\gg |\widetilde{E}(k)|\gg (p/2)^j$, where the implicit constant is
  absolute. Since~$E^{\sharp}(k)$ is cyclic, the size of its $2$-torsion
  subgroup is at most~$2$, so this symmetric Sidon set contains a Sidon
  set of size $\gg \demi |\widetilde{E}(k)|\gg (p/2)^j$ where the
  implicit constant is absolute.
\end{proof}

\section{Applications of Sidon sets in algebraic and arithmetic
  geometry}\label{sec-appl}

In this short section, we briefly recall some of the applications of
Sidon sets (even in cases where they are not very dense) in arithmetic
geometry.

(1) In work of Katz (see,
e.g.,~\cite[Th.\,2.8.1 and Th.\,7.9.6]{katz-esde}), the assumption that
the critical values of a polynomial, or the set of parameters of a
hypergeometric differential equation, form a Sidon set or a symmetric Sidon set in the additive group~$\Cc$ lead to computations of certain monodromy groups or differential Galois groups; here, even sets of~$3$ elements give
non-trivial results, and such computations in turn have a number of
important implications (see, for instance,~\cite{fresan-jossen}
and~\cite{ks2} for recent examples).

(2) The fact that the sets in Theorem~\ref{th-classification} are Sidon
sets or symmetric Sidon sets leads by our
work~\cite[Ch.\,7 and \S\,9.3]{ffk} to equidistribution results for
exponential sums over finite fields parameterized by characters of the
groups $J_{\mfm}(k)$. The (symmetric) Sidon property allows us
to compute the so-called fourth moment of the relevant tannakian group which controls the distribution properties of
the sums, and to almost determine it by means of Larsen's alternative~\cite{katz-larsen}. We emphasize again here that symmetric
Sidon sets arise just as naturally as Sidon sets, and that the size of
the sets is not particularly relevant. In the special case
$\mfm=0$, this was already used by Katz in 2010 to answer a
question of Tsimerman (unpublished, but see~\cite[Th.\,12.1]{ffk}).

(3) In another paper of Katz~\cite{katz-4-lectures}, it is shown how to
use Larsen's alternative to prove some of the key statements in
Deligne's second proof~\cite{weil2} of the Riemann hypothesis over
finite fields. The crucial moment computation (performed in
cohomological form in \hbox{\cite[p.\,120,\,Step\,5]{katz-4-lectures}})
relies ultimately (but implicitly) on the fact that the parabola
$(x,x^2)$ in~$k^2$ is a Sidon set in odd characteristic. As we pointed
out above, this is the case of the curve $C=\Pp^1$ and the modulus $\mfm=3(\infty)$ of
Theorem~\ref{th-classification}.

\section{Proofs}

We will now prove Theorem~\ref{th-classification}.  We first notice that
we may reduce to the case where the field $k$ is algebraically closed. Indeed,
let~$\bar{k}$ be an algebraic closure of~$k$, and let~$S$ be a subset
of~$J_{\mfm}(k)$. By definition, $S$ is a Sidon set
in~$J_{\mfm}(k)$ if and only if it is one
in~$J_{\mfm}(\bar{k})$. Similarly, if~$S$ is a symmetric Sidon
set in~$J_{\mfm}(k)$ with center~$a\in J_{\mfm}(k)$,
then it is one in~$J_{\mfm}(\bar{k})$, with the same center. But
also, suppose that~$S$ is a symmetric Sidon set
in~$J_{\mfm}(\bar{k})$ with
center~$a\in J_{\mfm}(\bar{k})$. Then either~$S$ is empty, or we
can find~$x$ and~$y$ in~$S$ (maybe equal) such that~$x=a-y$,
hence~$a=x+y\in J_{\mfm}(k)$, and it follows that~$S$ is a
symmetric Sidon set in~$J_{\mfm}(k)$ in all cases.

Thus we assume that~$k$ is algebraically closed and 
$\dim(J_{\mfm})\geq 2$.
  
We recall that if~$\mfm\geq \mfm'\geq 0$ and~$\delta$ is
a divisor of degree~$1$ with support disjoint from that
of~$\mfm$, then there is a commutative diagram
$$
\begin{tikzcd}
  J_{\mfm} \arrow[r,"f"]
  & J_{\mfm'}\\
  C\setminus \mfm\arrow[r,"i"]\arrow[u,"s"] &
  C\setminus \mfm'\arrow[u,"s"]
\end{tikzcd}
$$
where~$i$ is the inclusion and~$f$ is a group homomorphism
(see~\cite[V.3.12,\,Prop.\,6]{serre}). This implies that if
$s((C\setminus \mfm')(k))$ is a Sidon set in~$J_{\mfm'}(k)$,
then so is~$s((C\setminus \mfm)(k))$ in~$J_{\mfm}(k)$.

This means that we can reduce the proof of
Theorem~\ref{th-classification} to the following cases:
\begin{enumerate}
\item $g=0$ and $\deg(\mfm)=3$ (obtaining Sidon sets); 
\item $g=1$ and $\deg(\mfm)=2$ (obtaining symmetric Sidon sets); 
\item $g=1$ and $\deg(\mfm)\geq 3$ (obtaining Sidon sets); 
\item $g\geq 2$, $\mfm=0$ and $C$ not hyperelliptic (obtaining
  Sidon sets); 
\item $g\geq 2$ and~$C$ hyperelliptic (obtaining either Sidon sets or
  symmetric Sidon sets).
\end{enumerate}

\begin{proof}[Case \emph{(1)}]
  This corresponds to the three ``classical'' Sidon sets discussed in
  Section~\ref{sec-remarks}, but it is also straightforward to show that
  we obtain Sidon sets directly from the definition.

Thanks to the action of~$\PGL_2(k)$ on~$\Pp^1(k)$, it suffices to handle the cases
  $\mfm=3(\infty)$, $\mfm=2(\infty)+(0)$ and 
  $\mfm=(\infty)+(0)+(1)$. In any case, let $x_1$, \ldots, $x_4$
  be points of $(\Pp^1\setminus \mfm)(k)$ satisfying
  $s(x_1)+s(x_2)=s(x_3)+s(x_4)$, and assume that $x_1\notin
  \{x_3,x_4\}$. There is a unique rational function
  $\varphi\colon \Pp^1\to \Pp^1$ such that $\varphi(\infty)=1$ which
  vanishes at $x_1$ and $x_2$ and has poles at $x_3$ and $x_4$, namely
  $$
  \varphi=\frac{(X-x_1)(X-x_2)}{(X-x_3)(X-x_4)}.
  $$
  Thus the Sidon equation holds if and only if the divisor of $\varphi$ belongs to $P_{\mfm}(\Pp^1)$. Now:

  \begin{enumerate}
  \item If $\mfm=3(\infty)$, we need $v_{\infty}(\varphi-1)\geq
    3$, which is impossible because $\varphi$ has degree~$2$.
  \item If $\mfm=2(\infty)+(0)$, we need
    $(x_1x_2)/(x_3x_4)=\varphi(0)=1$ and $v_{\infty}(\varphi-1)\geq 2$;
    using the uniformizer $Z=1/X$ at infinity, the last condition is seen to be
    equivalent to $x_1+x_2=x_3+x_4$. But the two equations
    $x_1x_2=x_3x_4$ and $x_1+x_2=x_3+x_4$ imply that
    $\{x_1,x_2\}=\{x_3,x_4\}$, both sets being the roots of the same polynomial of degree~$2$. 
  \item If $\mfm=(\infty)+(0)+(1)$, we need the equalities 
    $(x_1x_2)/(x_3x_4)=\varphi(0)=1$ and
    $(1-x_1)(1-x_2)/((1-x_3)(1-x_4))=\varphi(1)=1$ to hold, and expanding we see
    that these two conditions are equivalent to $x_1x_2=x_3x_4$ and
    $x_1+x_2=x_3+x_4$ again, so we obtain once more
    $\{x_1,x_2\}=\{x_3,x_4\}$.
  \end{enumerate} This concludes the proof of the first case. 
\end{proof}

For case (2), recall that the group law of an elliptic curve $E \subset \Pp^2$ with neutral element~$0_E$ is characterized by the condition that if $\ell$ is
a line in~$\Pp^2$, then the sum of the intersection points of~$\ell$
and~$E$ (with multiplicity) is equal to~$0_E$. We will use the following.

\begin{lemma}\label{lm-geometry}
  Let~$E$ be an elliptic curve over~$k$ with
  $\widetilde{E}=E\setminus\{0_E\}$ given by a Weierstrass equation
  $$
  y^2+a_1xy+a_3y=f(x),\quad f\in k[X],\quad \deg(f)=3,\quad f\text{
    squarefree.}
  $$
  Let~$(x_1,\ldots,x_4)$ be points of~$\widetilde{E}(k)$ such that
  $x_1+x_2=x_3+x_4$. Then the line in the affine plane joining~$x_1$ to~$x_2$, or by
  convention the tangent line to the curve at~$x_1$ if~$x_1=x_2$, is
  parallel to the line joining~$x_3$ to~$x_4$, with the same convention,
  if and only if either~$\{x_1,x_2\}=\{x_3,x_4\}$, or $x_2=-x_1$. If the
  two lines are equal, then $\{x_1,x_2\}=\{x_3,x_4\}$.
\end{lemma}

\begin{proof}
  If~$x_2=-x_1$, then~$x_4=-x_3$, so that the two lines meet at the
  point at infinity, and hence are parallel in the affine plane.

  Conversely, we assume that~$x_2\not=-x_1$ (and hence $x_4\not=-x_3$).
  By the geometric description of the group law, the condition
  $x_1+x_2=x_3+x_4$ means that the two lines indicated meet in the
  \emph{affine} plane at the point $-(x_1+x_2)=-(x_3+x_4)\not=0_E$.
  Since the lines are parallel, they are equal. Then this common
  line~$\ell$ satisfies
  $$
  \ell\cap
  \widetilde{E}(k)=\{x_1,x_2,-(x_1+x_2)\}=\{x_3,x_4,-(x_3+x_4)\}.
  $$
  This implies~$x_1\in \{x_3,x_4\}$. Indeed, otherwise we would
  have~$x_1=-(x_3+x_4)=-(x_1+x_2)$, so
  that~$\ell\cap \widetilde{E}(k)=\{x_1,x_2\}$ and
  hence~$x_3\in \{x_1,x_2\}$, which yields a contradiction.
\end{proof}

\begin{proof}[Case \emph{(2)}]
  Let $p$ and $q$ be points of~$C(k)$ (not necessarily distinct) such
  that $\mfm=(p)+(q)$, and let $s$ denote an immersion
  $x\mapsto (x)-\delta$ from~$C\setminus\{p,q\}$ to~$J_{\mfm}$.
  Taking~$q$ as the origin of the group law, we can view~$C$ as an
  elliptic curve, which we denote by~$E$ and view as a
  smooth plane cubic curve.

  Let $x_1$, \ldots, $x_4$ in~$E\setminus \{p,q\}$ be solutions of
  $s(x_1)+s(x_2)=s(x_3)+s(x_4)$. We denote by $L_{12}$ (resp.~$L_{34}$)
  the line in the projective plane passing through $x_1$ and~$x_2$, or
  the tangent line to~$E$ at~$x_1$ if $x_1=x_2$ (resp. the line in the
  projective plane passing through $x_3$ and~$x_4$, or the tangent line
  to~$E$ at~$x_3$ if $x_3=x_4$).

  If $L_{12}=L_{34}$ then we deduce from Lemma~\ref{lm-geometry} that
  $\{x_1,x_2\}=\{x_3,x_4\}$. We assume that this is not the case. Let~$r$ be the intersection point of~$L_{12}$ and~$L_{34}$. Since the
  assumption implies that $x_1+x_2=x_3+x_4$, the description of the
  group law implies that~$r$ lies in $E(k)$. We denote by $\Pp^1_r$ the space
  of projective lines in~$\Pp^2$ passing through~$r$; it is an algebraic
  curve, isomorphic to $\Pp^1$, and we identify it with~$\Pp^1$ in such
  a way that~$L_{12}=0$ and $L_{34}=\infty$.
  % through
  % the map $g\colon \Pp^1_r\to \Pp^1$ so that $g(L_{12})=0$
  % and~$g(L_{34})=\infty$.

  Now define a map $E\setminus\{r\}\to \Pp^1_r$ by sending~$x$ to the
  line joining~$r$ and~$x$. This is an algebraic map and can be extended
  to a morphism $\varphi\colon E\to \Pp^1_r=\Pp^1$ such that
  $\varphi(r)$ is the tangent line to~$E$ at~$r$.

  We claim that $\varphi^{-1}(0)=\varphi^{-1}(L_{12})=\{x_1,x_2\}$ and
  $\varphi^{-1}(\infty)=\varphi^{-1}(L_{34})=\{x_3,x_4\}$. Indeed, the
  equality \hbox{$\varphi(x_1)=\varphi(x_2)=L_{12}$} holds by
  definition, and the only other~$x\in E$ that may map to~$L_{12}$ is
  $x=r$. But $\varphi(r)=L_{12}$ means that the tangent line at~$r$
  passes through~$x_1$ and through~$x_2$. Since~$E$ is a cubic, this is
  only possible if~$x_1=x_2$, and then if also $x_1=x_2=r$.

  The function $\varphi\colon E\to \Pp^1$ has divisor
  $(x_1)+(x_2)-(x_3)-(x_4)$. In particular, $\varphi(p)$ and
  $\varphi(q)$ are in~$k^{\times}$.  The definition of $J_{\mfm}$ shows
  that the equation $s(x_1)+s(x_2)=s(x_3)+s(x_4)$ is valid if and only
  if
  \begin{itemize}
  \item $p\not=q$ and $\varphi(p)=\varphi(q)$;
  \item or $p=q$ and $p$ is a zero of order $\geq 2$ of
    $\varphi(p)^{-1}\varphi-1$.
  \end{itemize}

  If $p\not=q$ then the condition $\varphi(p)=\varphi(q)$ means that the
  line joining $r$ to~$p$ is the same as the line joining $r$ to $q$,
  with the usual tangent convention if~$r=p$ or $r=q$. This means that
  $r$, $p$, $q$ are on the same line, so that $r+p+q=0$, hence
  $x_1+x_2=x_3+x_4=-r=p+q$. Conversely, the equality
  $\varphi(p)=\varphi(q)$ holds if $x_1+x_2=p+q=x_3+x_4$.

  If $p=q$, a moment's thought (or the crutch of writing down equations)
  shows that the condition that $p$ is a zero of order $\geq 2$ of
  $\varphi(p)^{-1}\varphi-1$ is valid if and only if the line joining
  $p$ and~$r$ is the tangent line to~$E$ at~$r$. This translates to
  $2p+r=0$, and hence to $x_1+x_2=x_3+x_4=2p$.

  Since the conditions above are equivalent with
  $s(x_1)+s(x_2)=s(x_3)+s(x_4)$, this concludes the proof of Case~(2).
\end{proof}  

\begin{proof}[Case~\emph{(3)}]
  Let~$x_1, x_2, x_2, x_4$ be $k$-points of~$C\setminus \mfm$ satisfying $s(x_1)+s(x_2)=s(x_3)+s(x_4)$ and~$x_1\notin \{x_3,x_4\}$.  For
  any effective divisor~$\mfm'=(p)+(q)$ of degree~$2$ such
  that~$\mfm\geq \mfm'$, projecting
  to~$J_{\mfm'}$ and applying case~(2), we see
  that~$x_1+x_2=x_3+x_4=p+q$. Varying~$p$ and~$q$ among $k$-points in the support of~$\mfm$, we see that~$\mfm$ is of the
  form~$\mfm=d(p)$ for some point~$p\in E(k)$ and~$d\geq 3$.  The
  Sidon equation then holds if and only if the unique
  function~$\varphi$ with divisor $(x_1)+(x_2)-(x_3)-(x_4)$
  with~$\varphi(p)=1$ is such that~$\varphi-1$ vanishes to
  order~$\geq d$ at~$p$. Since~$\varphi$ has degree~$\leq 2$, this is
  not possible. Thus, $s((C\setminus\mfm)(k))$ is a Sidon~set.
\end{proof}

\begin{proof}[Case~\emph{(4)}]
  Here~$C$ has genus~$\geq 2$, the modulus $\mfm$ is trivial and~$C$ is not
  hyperelliptic. Let $x_1, x_2, x_3, x_4$ be points in~$C(k)$ such
  that $s(x_1)+s(x_2)=s(x_3)+s(x_4).$ If $x_1\notin\{x_3,x_4\}$, this implies the existence of a rational
  function on $C$ with set of zeros $\{x_1, x_2\}$ and set of poles
  $\{x_3, x_4\}$, which corresponds to a morphism
  $f \colon C \to \Pp^1$ of degree at most~$2$. By definition, this is
  not possible unless $C$ is hyperelliptic (see,
  e.g.,~\cite[Def.~7.4.7]{Liu}), hence the result.
\end{proof}

\begin{proof}[Case~\emph{(5)}]
  Finally, we assume that~$C$ has genus $\geq 2$ and is
  hyperelliptic.

  We first assume that~$\mfm=0$. Let $x_1$, $x_2$, $x_3$, $x_4$
  be points in~$C(k)$ such that $s(x_1)+s(x_2)=s(x_3)+s(x_4)$ and
  $x_1\notin\{x_3,x_4\}$. As in the previous case, this implies the
  existence of a rational function on $C$ with set of zeros
  $\{x_1, x_2\}$ and set of poles $\{x_3, x_4\}$, which corresponds to a
  morphism $f \colon C \to \Pp^1$ of degree at most~$2$. Since there exists on~$C$ a unique morphism to~$\Pp^1$ of degree~$2$,
  up to automorphisms (see, e.g.,~\cite[Rem.\,7.4.30]{Liu}), the hyperelliptic involution~$i$ exchanges the points on the
  fibers of~$f$, whence $x_2=i(x_1)$ and $x_4=i(x_3)$.

  Conversely, for any $x_1$ and~$x_3$, a function $\varphi$ with divisor
  $(x_1)+(i(x_1))-(x_3)-(i(x_3))$ is given by $\varphi=\sigma\circ \pi$,
  where $\pi$ is the quotient $C\to C/i$ modulo the hyperelliptic
  involution and $\sigma \colon C/i\to \Pp^1$ is an isomorphism which
  maps $\pi(x_1)=\pi(i(x_1))$ to~$0$ and~$\pi(x_3)=\pi(i(x_3))$
  to~$\infty$.  This implies that $s(x_1)+s(i(x_1))=s(x_3)+s(i(x_3))$ holds 
  in~$J(k)$. In particular, the element $s(x)+s(i(x))$ in~$J(k)$ is
  independent of~$x\in C(k)$. If we denote it by~$a$, then we have
  $ a-s(x)=s(i(x))$ for all~$x$, so that~$s(C(k))$ is a symmetric Sidon
  set with center~$a$.

  Now assume that $\mfm$ has degree~$\geq 2$. If~$x_1$, \ldots,
  $x_4$ are such that $s(x_1)+s(x_2)=s(x_3)+s(x_4)$ and
  $x_1\notin \{x_3,x_4\}$, then by comparing with $\mfm'=0$, we
  deduce that we must have~$x_2=i(x_1)$ and~$x_4=i(x_3)$. Consider the
  function $\varphi=\sigma\circ \pi$ described previously.  The Sidon
  equation in~$J_{\mfm}$ requires that for all~$p$ in the
  support of~$\mfm$, the value of $\varphi(p)$ is the same,
  which can only be the case if~$\mfm$ has degree~$2$
  since~$\varphi$ itself has degree~$2$. As a first consequence, this
  means that the image of~$C$ is a Sidon set
  if~$\deg(\mfm)\geq 3$.

  On the other hand, if $\mfm=(p)+(q)$, then
  $\varphi(p)=\varphi(q)$ if and only if~$q=i(p)$. Thus the image of~$s$
  is again a Sidon set if this is not the case. Finally, if
  $\mfm$ is of this form, then we do
  have~$\varphi(p)=\varphi(i(p))$ (resp. $\varphi-\varphi(p)$ has a zero
  of order~$2$ at~$p$, if $p=i(p)$), so the image of~$s$ is then a
  symmetric Sidon set, with center the common value of~$s(x)+s(i(x))$
  for any~$x\in C\setminus \{p\}$.
\end{proof}

\section{The universal vector extension of an elliptic curve}

In the special case where the curve~$C$ is an elliptic curve~$E$, with origin~$0_E$, and the modulus~$\mfm$ is
$2(0_E)$, the generalized jacobian $E^{\sharp}=J_{2(0_E)}$ is also
classically known through other interpretations, related to its
identification with the so-called \emph{universal vector extension} of~$E$ (see
Coleman's paper~\cite[Prop.\,1.2]{coleman} for this relation).  We now
present, for the sake of variety, two proofs of Theorem~\ref{th-classification} in this
case, using these alternative descriptions.

\begin{proof}[Analytic proof over the complex numbers]
  We use an explicit description by Katz~\cite[App.\,C]{katz-eisenstein}
  of the universal extension and the morphism~$s$ (another concrete
  discussion by Corvajá, Masser and Zannier can be found
  in~\cite[\S\,3]{cmz}), which applies when~$k=\Cc$.

  Let $\Lambda\subset \Cc$ be a lattice so that
  $E(\Cc)\simeq\Cc/\Lambda$. Let $\wp$ denote the Weierstrass function
  for~$\Lambda$ and $\zeta$  the Weierstrass zeta
  function (so that $\zeta'=-\wp$; see, e.g.~\cite[Ch.\,20]{ww} for the
  classical theory of elliptic functions). Define a group homomorphism $\eta \colon \Lambda \to \Cc$ by setting 
   $$
  \eta(\ell)=\int_{p}^{p+\ell} \wp(z)dz=\zeta(p)-\zeta(p+\ell)
  $$
  where $p$ is any point in $\Cc \setminus \Lambda$ and the integration path avoids $\Lambda$. Let
  $\Lambda^{\sharp} \subset \Cc^2$ be the subgroup of elements of the
  form $(\ell,-\eta(\ell))$ for $\ell\in\Lambda$.
  
  Katz~\cite[p.\,300--301]{katz-eisenstein} shows that there is an isomorphism of complex Lie groups  
  $$
  E^{\sharp}(\Cc) \simeq \Cc^2/\Lambda^{\sharp}
  $$
 which is compatible with the projection to $E(\Cc)\simeq\Cc/\Lambda$. Under this identification, the embedding  $s\colon \widetilde{E}(\Cc)=(\Cc/\Lambda) \setminus\{0\}\to
  E^{\sharp}(\Cc)=\Cc^2/\Lambda^{\sharp}$ is given by the formula
  $$
  s(\alpha)=(\alpha,\zeta(\alpha))\bmod \Lambda^{\sharp}
  $$
  (see~\cite[Th.\,C.6\,(2)]{katz-eisenstein}).
  
  Now let $\alpha_1, \ldots, \alpha_4 \in \widetilde{E}(\Cc)$ satisfy 
  $$
  s(\alpha_1)+s(\alpha_2)=s(\alpha_3)+s(\alpha_4).
  $$ There exist representatives $(\alpha_i, \zeta(\alpha_i)) \in \Cc^2$ of $s(\alpha_i)$ such that the equation $\alpha_1+\alpha_2=\alpha_3+\alpha_4$ holds in $\Cc$. Then these representatives necessarily also satisfy 
  $$
  \zeta(\alpha_1)+\zeta(\alpha_2)=\zeta(\alpha_3)+\zeta(\alpha_4). 
  $$
  Let $x_i=(\wp(\alpha_i), \wp'(\alpha_i))=(a_i, b_i) \in \Cc^2$ be the points in $E(\Cc)$ corresponding to $\alpha_i \in \Cc/\Lambda$. 
  
  Let us first suppose that~$\alpha_1\not=\alpha_2$
  and~$\alpha_3\not=\alpha_4$. By a classical formula
  (see~\cite[p.\,304]{katz-eisenstein}, citing Whittaker and
  Watson~\cite[p.\,451,\,ex.\,2]{ww}), we have
  $$
  \zeta(\alpha_i)+\zeta(\alpha_j)=\zeta(\alpha_i+\alpha_j)+\frac{1}{2}\,
  \frac{\wp'(\alpha_i)-\wp'(\alpha_j)}{\wp(\alpha_i)-\wp(\alpha_j)},
  $$
  so the equation becomes
  $$
  \frac{\wp'(\alpha_1)-\wp'(\alpha_2)}{\wp(\alpha_1)-\wp(\alpha_2)}=
  \frac{\wp'(\alpha_3)-\wp'(\alpha_4)}{\wp(\alpha_3)-\wp(\alpha_4)},
  $$
  or in other words
  $$
  \frac{b_1-b_2}{a_1-a_2}= \frac{b_3-b_4}{a_3-a_4},
  $$
  or equivalently the line joining~$x_1$ to~$x_2$ is parallel to the
  line joining~$x_3$ to~$x_4$. We can then apply
  Lemma~\ref{lm-geometry}.

  In the remaining cases where~$\alpha_1=\alpha_2$
  (or~$\alpha_3=\alpha_4$) we argue as before with points
  $\alpha'_2\not=\alpha_1$ (or~$\alpha'_4\not=\alpha_3$) converging to
  $\alpha_2$ (or~$\alpha_4$) and deduce that the same condition as
  above holds where the slopes of the line joining~$\alpha_1$
  to~$\alpha_2$ is replaced where needed by the slopes of the tangent
  line at~$\alpha_1$ (and similarly with $\alpha_3$ and~$\alpha_4$).
\end{proof}

We have yet another argument in characteristic~$\geq 5$ using the
interpretation of $E^{\sharp}$ in terms of connections and
differentials of the third kind.

\begin{proof}[Proof with connections]
  We assume that~$k$ is algebraically closed of characteristic~$\geq
  5$. As explained by Katz~\cite{katz}, the points of $E^{\sharp}$ can
  be interpreted as isomorphism classes of pairs~$(\mcL,\nabla)$ consisting of a
 line bundle~$\mcL$ on~$E$ and a connection
  $\nabla \colon \mcL \to \mcL \otimes \Omega^1_E$.

  Since the characteristic is at least~$5$, we can describe an immersion
  $\widetilde{E}\to E^{\sharp}$ in that case as follows (see, for
  instance,~\cite[Lemma\,2.1]{katz}
  or~\cite[C.2,\,C.3]{katz-eisenstein}). First, view~$E$ as a plane
  cubic curve in short Weierstrass form $Y^2=f(X)$, and for
  $x=(a,b)\in \widetilde{E}$, let
  $$
  \omega_x=\frac{1}{2}\,\frac{Y+b}{X-a}\,\frac{dX}{Y}
  $$
  (this is a meromorphic differential $1$-form $\omega_x$ on~$E$ which
  has only simple poles at $x$ and~$0$, with residue~$1$ at~$x$ and
  residue~$-1$ at~$0$; classically, these are called ``differentials of
  the third kind'' on~$E$).  We then define $s(x)=(\mcL_x,\nabla_x)$,
  where $\mcL_{x}=\mcO((x)-(0))$ and the connection $\nabla_x$ is
  defined for a local section~$\xi$ of $\mcL_x$ by
  $$
  \nabla_x(\xi)=d\xi-\xi\omega_x.
  $$
  (which is easily checked to be well-defined).
  
  Let again $x_1$, \ldots, $x_4$ be elements of~$\widetilde{E}(k)$ such
  that $s(x_1)+s(x_2)=s(x_3)+s(x_4)$ holds. The points $s(x_1)+s(x_2)$ and $s(x_3)+s(x_4)$ are, respectively, the isomorphism classes of $(\mcL_{12},\nabla_{12})$ and $(\mcL_{34},\nabla_{34})$, with the notation 
  $$
  \mcL_{ij}=\mcO((x_i)+(x_j)-2(0_E)), \quad \nabla_{ij}(\xi)=d\xi-\xi(\omega_{x_i}+\omega_{x_j}).
  $$

  The equation implies as in the previous proof
  that~$x_1+x_2=x_3+x_4$, so that there
  exists a non-zero rational function~$\varphi \colon E \to \Pp^1$ with divisor
  $(x_1)+(x_2)-(x_3)-(x_4)$.  Multiplication by $1/\varphi$ is then an
  isomorphism $\mcL_{12}\to \mcL_{34}$, and there is no other
  isomorphism up to multiplication by a non-zero constant.
  
  For a local section $\xi$ of $\mcL_{12}$, we have
  $$
  \frac{1}{\varphi}\nabla_{12}(\xi)=\frac{1}{\varphi}
  (d\xi-\xi(\omega_{x_1}+\omega_{x_2})),
  $$
  while, on the other hand, we have
  $$
  \nabla_{34}(\xi/\varphi)=d(\xi\varphi^{-1})
  -\frac{\xi}{\varphi}(\omega_{x_3}+\omega_{x_4})=
  \frac{d\xi}{\varphi}-\xi\frac{d\varphi}{\varphi^2}
  -\frac{\xi}{\varphi}(\omega_{x_3}+\omega_{x_4}).
  $$
  
  The equation $s(x_1)+s(x_2)=s(x_3)+s(x_4)$ is therefore equivalent
  with the condition that the formula
  $$
  \frac{1}{\varphi} (d\xi-\xi(\omega_{x_1}+\omega_{x_2}))=
  \frac{d\xi}{\varphi}-\xi\frac{d\varphi}{\varphi^2}
  -\frac{\xi}{\varphi}(\omega_{x_3}+\omega_{x_4})
  $$
  holds for all local sections~$\xi$. This boils down to the equality
  $$
  \omega_{x_1}+\omega_{x_2}=\frac{d\varphi}{\varphi}
  +\omega_{x_3}+\omega_{x_4}
  $$
  of meromorphic differentials on~$E$.
  
  From the known poles and residues of these differentials, we see
  that the difference of the left and right-hand sides has no poles,
  and hence is a constant multiple of the holomorphic differential~$dx/y$. To determine the constant,
  say~$\alpha$, and check when it vanishes, we look close
  to~$0_E$. Using the uniformizer~$x/y$, the properties of the
  $\omega_{x_i}$ show that $\alpha=0$ if and only
  if~$d\varphi/\varphi$ is~$0$ at~$0_E$.

  Recall that we view~$\widetilde{E}$ as a plane Weierstrass curve, and
  as before we can take $\varphi=\ell_{12}/\ell_{34}$, where
  $\ell_{ij}=\alpha_{ij}x+\beta_{ij}y+\gamma_{ij}$ defines the affine
  line joining~$x_i$ to~$x_j$ in the plane (resp. the tangent line
  to~$E$ at~$x_i$ if~$x_i=x_j$). Thus
  $$
  \frac{d\varphi}{\varphi}=\frac{d\ell_{12}}{\ell_{12}}-
  \frac{d\ell_{34}}{\ell_{34}},\quad\quad
  \frac{d\ell_{ij}}{\ell_{ij}}=\frac{\alpha_{ij}dx}{\ell_{ij}}
  +\frac{\beta_{ij}dy}{\ell_{ij}}.
  $$
  \par
  Let~$\pi=x/y$ be a uniformizer at~$0_E$.  Then
  $$
  \frac{\alpha_{ij}dx}{\ell_{ij}}=
  \Bigl(\frac{\alpha_{ij}}{\beta_{ij}}+O(\pi)\Bigr)\frac{dx}{y},
  $$
  while~$dy/\ell_{ij}$ has a pole at~$0$. Thus the contributions of
  these last terms must cancel, as well as those involving~$O(\pi)$,
  and we find that
  $$
  (\omega_{x_1}+\omega_{x_2})-\Bigl(\frac{d\varphi}{\varphi}
  +\omega_{x_3}+\omega_{x_4}\Bigr)=
  \Bigl(\frac{\alpha_{12}}{\beta_{12}}-
  \frac{\alpha_{34}}{\beta_{34}}\Bigr)\frac{dx}{y}.
  $$

  We therefore have~$\alpha=0$ if and only if
  $\alpha_{12}/\beta_{12}-\alpha_{34}/\beta_{34}=0$, which once more
  means that the lines~$\ell_{12}$ and~$\ell_{34}$ are parallel,
  allowing us to conclude using Lemma~\ref{lm-geometry}.
\end{proof}

\section{Final questions}

We conclude with some natural questions arising from this work:

\begin{enumerate}
\item Are there other interesting examples of Sidon sets arising from algebraic geometry? In~\cite[\S~12.2]{ffk}, we point
  out that a classical construction related to smooth cubic threefolds
  leads to a morphism~$s$ from a surface~$S$ to an abelian variety~$A$
  of dimension~$5$ such that the equation $s(\ell_1)+s(\ell_2)=x$
  admits generically either zero or six solutions for given~$x\in A$.
\item How close are the Sidon sets that we construct from being
  \emph{maximal}? In particular, can one embed one of the
  fairly dense examples into even larger Sidon sets?
\item What are the most general statements of existence of Sidon sets in
  ``abstract'' finite abelian groups that can be deduced from these
  constructions?
\end{enumerate}

\bibliographystyle{abbrv}
\bibliography{sidon}

\end{document}